\documentclass[12pt]{article}
\usepackage{mathrsfs}
\usepackage{amsmath}
\usepackage{amssymb}
\usepackage{amsfonts}
\usepackage{wasysym}
\usepackage{verbatim}
\usepackage[amsmath,thmmarks]{ntheorem}

\def\H{{\cal H}}
\def\F{\mathscr{F} }

\def\R{\mathbb{R}}

\def\T{\mathbb{T}}

\def\ba{\begin{array}}
\def\ea{\end{array}}
\def\be{\begin{enumerate}}
\def\ee{\end{enumerate}}
\def\bi{\begin{itemize}}
\def\ei{\end{itemize}}
\def\bc{\begin{cases}}
\def\ec{\end{cases}}
\def\beq{\begin{equation}}
\def\eeq{\end{equation}}
\def\beqs{\begin{equation*}}
\def\eeqs{\end{equation*}}
\def\beqa{\begin{eqnarray}}
\def\eeqa{\end{eqnarray}}
\def\beqas{\begin{eqnarray*}}
\def\eeqas{\end{eqnarray*}}
\def\bmul{\begin{multline}}
\def\emul{\end{multline}}
\def\bmuls{\begin{multline*}}
\def\emuls{\end{multline*}}
\def\bg{\begin{gather}}
\def\eg{\end{gather}}
\def\bgs{\begin{gather*}}
\def\egs{\end{gather*}}

\theorembodyfont{\normalfont\sl}
\newtheorem{thm}{Theorem}[section]
\newtheorem{cor}[thm]{Corollary}
\newtheorem{lem}[thm]{Lemma}
\newtheorem{prop}[thm]{Proposition}

\numberwithin{equation}{section}

\begin{document}

\baselineskip=24pt
%
%
%
%
%
%
%
%
%
%
%
%
%
%
%
%
%
%
%
%
%
%
%
%
%
%

\newpage

\setcounter{page}{1}

\title{
\baselineskip=24pt
 \bf Global well-posedness for the Benjamin \\[.5ex]
 \bf equation in low regularity\footnote{This work is
 supported by National Natural Science Foundation of China
 under grant numbers 10471047 and 10771074.}}

\bigskip

\author{
\baselineskip=24pt
 \normalsize Yongsheng Li \quad and \quad
             Yifei Wu\footnote{Email: yshli@scut.edu.cn (Y. S. Li) and
             yerfmath@yahoo.cn (Y. F. Wu)}\\[1ex]
 \normalsize Department of Mathematics,
             South China University of Technology, \\[1ex]
 \normalsize Guangzhou, Guangdong 510640, P. R. China
 }

\bigskip

\date{}

\bigskip

\bigskip

\leftskip 1cm
\rightskip 1cm

\maketitle

\noindent
 {\small{\bf Abstract}\quad
 \baselineskip=24pt
In this paper we consider the initial value problem of the Benjamin
equation
$$
\partial_{t}u+\nu \H(\partial^2_xu)
     +\mu\partial_{x}^{3}u+\partial_xu^2=0,
$$
where $u:\R\times [0,T]\mapsto \R$, and the constants $\nu,\mu\in
\R,\mu\neq0$. We use the I-method to show that it is globally
well-posed in Sobolev spaces $H^s(\R)$ for $s>-3/4$. Moreover, we
use some argument to obtain a good estimative for the lifetime of
the local solution, and employ some multiplier decomposition
argument to construct the almost conserved quantities.}
\bigskip


\bigskip

\noindent
 {\small{\bf  MR(2000) Subject Classification:}\quad 35Q53}

\bigskip

\bigskip

\leftskip 0cm
\rightskip 0cm

\normalsize

\baselineskip=24pt

\section{Introduction}
We consider the initial value
problem (IVP) for the Benjamin equation
\renewcommand{\arraystretch}{2}
\begin{eqnarray}\label{kbo}
  && \partial_{t}u+\nu \H(\partial^2_xu)
     +\mu\partial_{x}^{3}u+\partial_xu^2=0,
     \qquad u:\R\times [0,T]\mapsto \R,\\
  && u(x,0)=u_0(x)\in H^s(\R),\label{1.2}
\end{eqnarray}
where the constants $\nu,\mu\in
\R,\mu\neq0$, $\H$ denotes the Hilbert transform defined by
$$
  \H f(x)=\mathrm{P.V.}\dfrac{1}{\pi}\int\dfrac{f(x-y)}{y}\,dy,
$$
i.e. $\widehat{\H f}(\xi)=-i\, \makebox{sgn}(\xi)\hat{f}(\xi)$.
The hat \^{} denotes the Fourier transform.

The equation (\ref{kbo}) was introduced by Benjamin \cite{B} to
describe a class of the intermediate waves in the stratified fluid.
The equation is also applied in other fluids.
Recently, Gleeson, Hammerton, Papageorgiou and Vanden-Broeck \cite{GHPV}
found a new application in interfacial electrohydrodynamics, they considered
the waves on a layer of finite depth with the influence of vertical electric fluid
and derived a Benjamin equation.
The linear part of  (\ref{kbo}) is formed by combining the linear
parts of the Korteweg-de Vries (KdV) and Benjamin-Ono equation together,
so (\ref{kbo}) is often called the Korteweg-de Vries--Benjamin-Ono equation.

The Benjamin equation was studied by several authors on the low regularity theories.
Linares \cite{Li} proved the global well-posedness of IVP of
(\ref{kbo})-(\ref{1.2}) in $L^2(\R)$; Kozono, Ogawa and Tanisaka
\cite{KOT} showed the local well-posedness in negative index space
$H^s(\R)$ with $s>-3/4$ by the argument in \cite{Bourgain} and \cite{KPV1}.
For the modified KdV--BO equation with a
trilinearity, Guo and Huo \cite{GH1} proved the local well-posedness
in $H^s(\R)$ for $s>1/4$ (who also proved the local well-posedness
for the IVP of Benjamin equation when $s\geq-1/8$). In \cite{GH2}, they studied
further the existence and regularity of the global attractor of the
damped, forced Benjamin equation in $L^2(\R)$.

We consider the global well-posedness for (\ref{kbo})-(\ref{1.2}) in
$H^s(\R)$ for $s<0$ in this work. The multilinear harmonic analysis
($I$-method) is introduced by Colliander, Keel, Staffilani, Takaoka and
Tao (see \cite{CKSTT4}, \cite{CKSTT} for examples)
to study the global well-posedness theory in low
regular space. It is mainly dependent on an almost conservation law and
the iteration which is based on the former and the local existence
intervals. If the solution of an equation lacks the
scale invariance, unlike the KdV equation ($\nu=0,\mu=1$ in
(\ref{kbo})), then the threshold of the global well-posedness in
$H^s(\R)$ is decided by two ingredients: the increment of the almost conserved
quantities and the lifetime in the local theory. One of the argument
here is to lengthen the
lifetime of the local existence by establishing a variant
local well-posedness result, which is based on a special bilinear
estimate (see Proposition 3.2 below).
We believe that these techniques are of independent
interest and may be useful for other equations which lacks the solution
of scale invariance.
Indeed, we have succeed in applying this argument to establish the global
well-posedness results of NLS-KdV system in $H^s(\R)\times H^s(\R) $ for $s>1/2$,
which improve the results in \cite{P}.
Moreover, in order to establish global well-posedness in $H^s(\R)$ for any
$s>\dfrac{3}{4}$,
it also requires the development of the techniques
in \cite{CKSTT}, because
of the complexity of the linear principle operator which makes some troubles to
give the pointwise estimates on the multipliers (for more detailed explanations, see
Section 4). In this purpose, we employ some multiplier decomposition
argument, which is featured by convenient operation.
More precisely, we split the multiplier
($M_4$, defined in Section 4) into two parts ($\bar{M}_4, \tilde{M}_4$),
then we remain $\bar{M}_4$, and deduce $\tilde{M}_4$
into a higher order cancelation by introducing the next generation modified
energy.

{\bf Some notations.} We use $A\lesssim B$ or $B\gtrsim A$ to denote
the statement that $A\leq CB$ for some large constant $C$ which may
vary from line to line, and may depend on the coefficients such as
$\mu,\nu$ and the index $s$. When it is necessary, we will write the
constants by $C_1,C_2,\cdots$ to see the dependency relationship. We
use $A\ll B$, or sometimes $A=o(B)$ to denote the statement $A\leq C^{-1}B$,
and use $A\sim
B$ to mean $A\lesssim B\lesssim A$. The notation $a+$ denotes
$a+\epsilon$ for any small $\epsilon$, and $a-$ for $a-\epsilon$.
$\langle\cdot\rangle=(1+|\cdot|^2)^{1/2}$ and
$D_x^\alpha=(-\partial^2_x)^{\alpha/2}$. We use $\|f\|_{L^p_xL^q_t}$
to denote the mixed norm
$\Big(\displaystyle\int\|f(x,\cdot)\|_{L^q}^p\
dx\Big)^{\frac{1}{p}}$. Moreover, we denote $\F_x$ to be
the Fourier transform corresponding to the variable $x$.
We define the Fourier restriction operators $P^l$, $P_l$ respectively as
$$
P^lf(x)=\displaystyle\int_{|\xi|\geq l}
e^{ix\xi}\hat{f}(\xi)\,d\xi,\quad
P_lf(x)=\displaystyle\int_{|\xi|\leq l}
e^{ix\xi}\hat{f}(\xi)\,d\xi
$$
for any $l>0$. Finally, we denote
the constant $a=2\max\Big(1,\Big|\dfrac{2\nu}{3\mu}\Big|\Big)$, it
will be often used in the analysis.

Now we introduce some definitions before presenting our main
result.

For $s,b\in \R $, define the Bourgain space $X_{s,b}$ to be the
closure of the Schwartz class under the norm
\begin{equation}
\|u\|_{X_{s,b}}\equiv \left(\int\!\!\!\!\int
\langle\xi\rangle^{2s}\langle\tau-\phi(\xi)\rangle^{2b}|\hat{u}(\xi,\tau)|^2\,d\xi
d\tau\right)^{\frac{1}{2}},\label{X}
\end{equation}
where
$\phi(\xi)=-\nu\xi|\xi|+\mu\xi^3$ is the phase function of the semigroup
generated by the linear Benjamin equation.

For an interval $\Omega$, we define $X_{s,b}^{\Omega}$ to be the
restriction of $X_{s,b}$ on $\R\times\Omega$ with the norm \beq
\|u\|_{X_{s,b}^\Omega}=\inf\{\|U\|_{X_{s,b}}:U|_{t\in\Omega}=u|_{t\in\Omega}\}.
\label{X1}
\eeq When $\Omega=[-\delta,\delta]$, we write $X_{s,b}^\Omega$ as
$X_{s,b}^\delta$.
By the limiting argument we see that, for every $u\in
X_{s,b}^{\delta}$, there exists an extension $\tilde{u}\in X_{s,b}$
such that $\tilde{u}=u$ on $\Omega$ and
$\|u\|_{X_{s,b}^{\delta}}=\|\tilde{u}\|_{X_{s,b}}$  (see
also \cite{G}).

Let $s<0$ and $N\gg 1$ be fixed,  the Fourier
multiplier operator $I_{N,s}$ is defined as
\beq
\widehat{I_{N,s}u}(\xi)=m_{N,s}(\xi)\hat{u}(\xi),\label{I}
\eeq
where the multiplier $m_{N,s}(\xi)$ is a smooth, monotone function
satisfing $0<m_{N,s}(\xi)\leq 1$ and
 \beq m_{N,s}(\xi)=\Biggl\{
\begin{array}{ll}
1,&|\xi|\leq N,\\
N^{-s}|\xi|^s,&|\xi|>2N.\label{m}
\end{array}
\eeq
Sometimes we denote $I_{N,s}$ and $m_{N,s}$ as $I$ and $m$
respectively for short if there is no confusion.

It is obvious that the operator $I_{N,s}$ maps $H^s(\R)$ into
$L^2(\R)$ with equivalent norms for any $s<0$. More precisely, there
exists some positive constant $C$ such that \beq
    C^{-1}\|u\|_{H^s}\leq \|I_{N,s}u\|_{L^2}
\leq
    CN^{-s}\|u\|_{H^s}.\label{II}
\eeq Moreover, $I_{N,s}$ can be extended to a map (still denoted by
$I_{N,s}$) from $X_{s,b}$ to $X_{0,b}$ which satisfies
$$
    C^{-1}\|u\|_{X_{s,b}}\leq\|I_{N,s}u\|_{X_{0,b}}
\leq CN^{-s}\|u\|_{X_{s,b}}
$$
for any $s<0,b \in \R$.

Now we are ready to state our main result.
\begin{thm}
The IVP (\ref{kbo})-(\ref{1.2}) is globally well-posed in $H^s(\R)$
for $s>-\dfrac{3}{4}$. More precisely, for any $u_0 \in H^s(\R)$
with $s>-\dfrac{3}{4}$ and $T>0$, (\ref{kbo})-(\ref{1.2}) has a
unique solution $u\in X_{s,\frac{1}{2}+}^T\subset C([0,T],H^s(\R))$,
and the solution map $u_0\mapsto u[u_0]$ is continuous from
$H^s(\R)$ to $X_{s,\frac{1}{2}+}^T$.
\end{thm}

The rest of this article is organized as follows. In Section 2, we
present some preliminary estimates. In Section 3, we will
give a key bilinear estimate and establish the variant local
well-posedness result. In Section 4, we use the I-method to prove Theorem 1.1.

\vspace{0.5cm}
\section{Some Preliminary Estimates}

As it's well-known, the corresponding linear equation of (\ref{kbo})
\beq
 \partial_{t}u+\nu \H(\partial^2_xu)+\mu\partial_{x}^{3}u=0, \qquad x,\,t\in \R
 \label{lkbo}
\eeq generates a unitary group $\{S(t)\}_{t\in \R}$ on $L^2(\R)$
such that $u=S(t)u_0$ solves (\ref{lkbo})-(\ref{1.2}). It is also
defined explicitly by spatial Fourier transform as
$$
\widehat{S(t)u_0}(\xi)\triangleq e^{it\phi(\xi)}\widehat{u_0}(\xi).
$$
The first part of estimates in this section are some standard Strichartz
estimates concerning this group. We remark that some Fourier
restriction operators shall be used in these estimates because of
the presence of the nontrivial zero points of the phase function
$\phi(\xi)$, which is different from the KdV equation. See
\cite{GH1} for details.
\begin{lem}
 For $u_0\in L^2(\R)$,
\begin{eqnarray}
              \|D_xS(t)P^au_0\|_{L^\infty_x L^2_t}
  & \lesssim &
              \|u_0\|_{L^2},\label{TSE1} \\
              \left\|D_x^{-\frac{1}{4}}S(t)P^au_0\right\|_{L^4_x L^\infty_t}
  & \lesssim &
              \|u_0\|_{L^2}, \label{TSE2}\\
              \|D_x^{\alpha}S(t)P^au_0\|_{L^p_x L^q_t}
  & \lesssim &
              \|u_0\|_{L^2},\label{TSE3}\\
              \|S(t)u_0\|_{L^8_x L^8_t}
  & \lesssim &
              \|u_0\|_{L^2},\label{TSE4}
\end{eqnarray}
where
$\dfrac{1}{p}=\dfrac{1}{5}(1-\alpha)$,
$\dfrac{1}{q}=\dfrac{1}{10}(4\alpha+1)$, for any
$\alpha\in [-\dfrac{1}{4},1]$.
\end{lem}
\begin{proof}
See \cite{GH1} for the proof of (\ref{TSE1}), (\ref{TSE2}) and
(\ref{TSE4}) (see also \cite{KPV}). (\ref{TSE3}) follows by interpolation between
(\ref{TSE1}) and (\ref{TSE2}).\hfill$\Box$
\end{proof}

\begin{lem}
 Let $\alpha,\,p,\,q\,$
 be as in Lemma 2.1. For $F\in X_{0,\frac{1}{2}+}$,
\beq \|D_x^{\alpha}P^aF\|_{L^p_x L^q_t}\lesssim
\|F\|_{X_{0,\frac{1}{2}+}}. \label{XE1} \eeq
\end{lem}
\begin{proof}We shall omit the details here, since the argument is well-known
(see \cite{KPV4}).
  \hfill$\Box$
\end{proof}

By interpolating between (\ref{XE1}) and the following equality
\begin{equation}
 \|F\|_{L^2_ {xt}}=
 \|F\|_{X_{0,0}},\label{L2}
\end{equation}
we can generalize (\ref{XE1}) as below.

\begin{lem}For any
$\theta\in [0,1]$, $\alpha\in [-\dfrac{\theta}{4},\theta]$ and
$F\in X_{0,\frac{\theta}{2}+}$,  we have
\beq
 \|D_x^{\alpha}P^aF\|_{L^p_x
L^q_t}\lesssim \|F\|_{X_{0,\frac{\theta}{2}+}},\label{XE2}
\eeq
where
$\dfrac{1}{p}=\dfrac{1}{2}-\dfrac{1}{5}\alpha-\dfrac{3}{10}\theta$,
$\dfrac{1}{q}=\dfrac{1}{2}+\dfrac{2}{5}\alpha-\dfrac{2}{5}\theta$.
\end{lem}

Similarly, combining (\ref{TSE4}) with (\ref{L2}), we have
\begin{lem}For
$\rho\geq \dfrac{2(q-2)}{3q}$, $q\in [2,8]$ and $F\in X_{0,\rho+}$,
we have \beq
 \|F\|_{L^q_{xt}}\lesssim \|F\|_{X_{0,\rho+}}.\label{XE3}
\eeq
\end{lem}

 At the end of this part, we introduce an operator which first
appeared in \cite{G} (a similar argument was used in \cite{CKS1}).
Define the bilinear Fourier integral operator
$I^s(u,v)$ by
\begin{equation}
\widehat{I^s(u,v)}(\xi,\tau)=\displaystyle\int_{\star}|\phi'(\xi_1)-\phi'(\xi_2)|^s
\hat{u}(\xi_1,\tau_1)\hat{v}(\xi_2,\tau_2),\label{Is}
\end{equation}
where $\displaystyle\int_\star =\int_{\stackrel{\xi_1+\xi_2=\xi,}{
\tau_1+\tau_2=\tau}}\,d\xi_1 d\tau_1$.
Now we give some estimates on this operator.
\begin{lem}
Let $I^{\frac{1}{2}}$ be defined by (\ref{Is}), then for any $u,v\in
X_{0,\frac{1}{2}+}$,
\begin{equation}
\Big\|I^{\frac{1}{2}}(u,v)\Big\|_{L^2_{xt}} \lesssim
\|u\|_{X_{0,\frac{1}{2}+}}\,\|v\|_{X_{0,\frac{1}{2}+}}.\label{IE}
\end{equation}
\end{lem}
\begin{proof}
We use the argument in \cite{CKS1} to prove the result. By the
definition (\ref{Is}) and the duality,
the left-hand side of (\ref{IE}) is equal to
\begin{equation}
\sup_{\|h\|_{L^2}\leq 1}\displaystyle\int
     |\phi'(\xi_1)-\phi'(\xi_2)|^{\frac{1}{2}}
     \hat{h}(\xi_1+\xi_2,\tau_1+\tau_2)
     \hat{u}(\xi_1,\tau_1)\hat{v}(\xi_2,\tau_2)
     \,d\xi_1 d\xi_2 d\tau_1 d\tau_2.\\\label{216}
\end{equation}
First, we change variables by setting
$$
\tau_1=\lambda_1+\phi(\xi_1),\quad \tau_2=\lambda_2+\phi(\xi_2),
$$
then (\ref{216}) is changed into
\renewcommand{\arraystretch}{2}
\begin{eqnarray}
 &&  \sup_{\|h\|_{L^2}\leq 1}\displaystyle\int
     |\phi'(\xi_1)-\phi'(\xi_2)|^{\frac{1}{2}}
     \hat{h}(\xi_1+\xi_2,\lambda_1+\lambda_2+\phi(\xi_1)+\phi(\xi_2))\nonumber\\
 &&  \qquad \qquad \cdot \hat{u}(\xi_1,\lambda_1+\phi(\xi_1))
     \hat{v}(\xi_2,\lambda_2+\phi(\xi_2))\, d\lambda_1 d\lambda_2 d\xi_1 d\xi_2.
     \label{2.16}
\end{eqnarray}
We change variables again as follows. Let
\begin{equation}
 (\eta,\omega)=T(\xi_1,\xi_2),\label{transform}
\end{equation}
where \beqs
\begin{split}
 \eta & = T_1(\xi_1,\xi_2)=\xi_1+\xi_2,\\
 \omega & = T_2(\xi_1,\xi_2)=\lambda_1+\lambda_2+\phi(\xi_1)+\phi(\xi_2).
\end{split}
\eeqs Then the Jacobian $J$ of this transform satisfies
$$
|J|=|\phi'(\xi_1)-\phi'(\xi_2)|.
$$
Define
$$
H(\eta,\omega,\lambda_1,\lambda_2)= \hat{u}\hat{v}\circ
T^{-1}(\eta,\omega,\lambda_1,\lambda_2),
$$
then, by eliminating $|J|^{\frac{1}{2}}$ with
$|\phi'(\xi_1)-\phi'(\xi_2)|^{\frac{1}{2}}$,
(\ref{2.16}) has a bound of
\begin{equation}
    \sup_{\|h\|_{L^2}\leq 1}\displaystyle\int \hat{h}(\eta,\omega)
    \cdot \dfrac{ H(\eta,\omega,\lambda_1,\lambda_2)}
    {|J|^{\frac{1}{2}}}
    \,d\eta d\omega d\lambda_1 d\lambda_2.\label{2.17}
\end{equation}
Further, by H\"{o}lder's inequality we have
\renewcommand{\arraystretch}{2}
\begin{eqnarray*}
(\ref{2.17})&\leq&
      \sup_{\|h\|_{L^2}\leq 1}\big\|\hat{h}\big\|_{L^2_{\eta\omega}}
      \cdot\displaystyle\int
      \Big(\int \dfrac{|H(\eta,\omega,\lambda_1,\lambda_2)|^2}{|J|}
      \,d\eta \omega\Big)^{\frac{1}{2}}
      \,d\lambda_1 d\lambda_2\nonumber\\
&\lesssim&
      \displaystyle\int
      \|\hat{u}(\xi_1,\lambda_1+\phi(\xi_1))\|_{L^2_{\xi_1}}
      \,d\lambda_1
      \cdot
      \displaystyle\int
      \|\hat{v}(\xi_2,\lambda_2+\phi(\xi_2))\|_{L^2_{\xi_2}}
      \,d\lambda_2\\
&\lesssim&
      \|u\|_{X_{0,\frac{1}{2}+}}\,\|v\|_{X_{0,\frac{1}{2}+}},\nonumber
\end{eqnarray*}
where we have employed the inverse transform of (\ref{transform}) in
the second step, triangle and H\"{o}lder's inequalities in the third
step. \hfill$\Box$
\end{proof}

When $s=0$, by (\ref{XE3}) we have
\begin{equation}
\|uv\|_{L^2_{xt}}\leq
\|u\|_{L^8_{xt}}\|v\|_{L^{\frac{8}{3}}}\lesssim
\|u\|_{X_{0,\frac{1}{2}+}}\, \|v\|_{X_{0,\frac{1}{6}+}}.
\label{2.18}
\end{equation}
Interpolation between (\ref{IE}) and (\ref{2.18}), we have
\begin{cor}
Let $I^{s}$ be defined by (\ref{Is}), then for any $s\in
[0,\dfrac{1}{2}]$, $\tilde{b}\geq \dfrac{1}{6}+\dfrac{2}{3}s$,
\begin{equation}
\|I^{s}(u,v)\|_{L^2_{xt}} \lesssim
\|u\|_{X_{0,\frac{1}{2}+}}\>
\|v\|_{X_{0,\tilde{b}+}}.
\label{IE1}
\end{equation}
\end{cor}

It's easy to verify that Lemma 2.5 and Corollary 2.6 still hold if one replaces
the operator $I^s$ by the one (still denoted by $I^s$) defined as
\begin{equation}\label{Is1}
\widehat{I^s(u,v)}(\xi_1,\tau_1)
=\displaystyle\int_{\star}|\phi'(\xi)-\phi'(\xi_2)|^{\frac{1}{2}}
\hat{u}(\xi,\tau)\hat{v}(\xi_2,\tau_2).
\end{equation}

We now continue to present some estimates of the group $\{S(t)\}$ in
$X_{s,b}$. We denote $\psi(t)$ to be an even smooth characteristic
function of the interval $[-1,1]$.
\begin{lem}(\cite{KPV4}) Let $\delta\in (0,1)$, $s\in \R$, then the following
estimates hold: \bi
\item[{\rm(i)}]
           $\|u\|_{C_t^0 (H_x^s;\R)}
            \lesssim\|u\|_{X_{s,b}}$,
           $\forall$ $b\in (\dfrac{1}{2},1],u\in X_{s,b}$;

\item[{\rm(ii)}]
           $\|\psi(t)S(t)u_0\|_{X_{s,b}}
            \lesssim\|u_0\|_{H^s}$,
           $\forall$ $b\in (\dfrac{1}{2},1],u_0\in H^s(\R)$;

\item[{\rm(iii)}]
          $\left\|\psi(t)\displaystyle\int^t_0 S(t-s)F(s)\,ds
          \right\|_{X_{s,b}}
          \lesssim
          \left\|F\right\|_{X_{s,b-1}}$,
          $\forall$ $b\in (\dfrac{1}{2},1],F\in X_{s,b-1}$;

\item[{\rm(iv)}]
          $\|\psi(t/\delta)f\|_{X_{s,b'}}
          \lesssim
          \delta^{b-b'}\|f\|_{X_{s,b}}$,
          for $0\leq b'\leq b<\dfrac{1}{2}$.
\ei
\end{lem}

\noindent{\it Remark.} See \cite{CKS1} for the proof of Lemma 2.7 (iv) when $b=b'$.
\begin{cor}Let $b\in [0,\dfrac{1}{2})$, $\delta\in (0,1)$, then
\beq
 \|u\|_{X_{s,b}^\delta}\lesssim \delta^{(\frac{1}{2}-b)-}
 \|u\|_{X_{s,\frac{1}{2}}^\delta}.\label{cor}
\eeq
\end{cor}
\begin{proof}Let $\tilde{u}$ be the extension of
$u\in X_{s,\frac{1}{2}}^\delta$ defined after
Definition 1.1. By Lemma 2.7 (iv),
$$
\|u\|_{X_{s,b}^\delta}\leq
\|\psi(t/\delta)\tilde{u}\|_{X_{s,b}}\lesssim\delta^{(\frac{1}{2}-b)-}
\|\tilde{u}\|_{X_{s,\frac{1}{2}}}
=\delta^{(\frac{1}{2}-b)-}\|u\|_{X_{s,\frac{1}{2}}^\delta}.
$$
This completes the proof of the corollary.  \hfill$\Box$
\end{proof}

\vspace{0.5cm}
\section{A Bilinear Estimate and the Local Well-posedness}

In this section, we will establish a variant local well-posedness
result as follows.
\begin{prop}
Let $s> -3/4$, then IVP (\ref{kbo})-(\ref{1.2}) is locally well-posed
in $H^s(\R)$.  Moreover, the solution exists on the interval
$[0,\delta]$ with the lifetime
\begin{equation}
 \delta\sim\|I_{N,s}u_0\|^{-2-}_{L^2}\label{delta}
\end{equation}
 when $N\geq N_0$ for some large number $N_0$ such that
 \beq
N_0^{-\frac{3}{2}+}\cdot \|I_{N_0,s}u_0\|_{L^2} \sim 1,\label{cond}
 \eeq
further, the solution satisfies the estimate \beq
    \|I_{N,s}u\|_{X_{0,\frac{1}{2}+}^\delta}
    \lesssim
    \|I_{N,s}u_0\|_{L^2}.\label{LSE}
\eeq
\end{prop}
\noindent{\it Remark.} The condition (\ref{cond}) is reasonable by
taking $ N_0\gtrsim
 \|u_0\|_{H^s}^{\frac{2}{3+2s}+},$
since $\|Iu_0\|_{L^2}\leq CN^{-s}\|u_0\|_{H^s}$.

 Compared with the standard local well-posedness result, this
proposition is established for adapting to the I-method. It gives
the estimates on the lifetime and the solution under the
$X_{s,\frac{1}{2}+}^\delta$- norm, with the operator $I_{N,s}$. The
proposition is based on the following bilinear estimates.
\begin{prop}Let $s\in (-\dfrac{3}{4},0),b=\dfrac{1}{2}+$,
$\delta\in (0,1)$, $N\gg 1$, then
for any $u,v \in X_{s,b}^\delta$,
\begin{equation}
 \|\psi(t/\delta)\>\partial_xI(\tilde{u}\tilde{v})\|_{X_{0,b-1}}\lesssim
 (\delta^{\frac{1}{2}-}+N^{-\frac{3}{2}+})\|Iu\|_{X_{0,b}^\delta}
 \|Iv\|_{X_{0,b}^\delta}, \label{Thm}
\end{equation}
where $I=I_{N,s}$, $\tilde{u}$ and $\tilde{v}$ are the extensions of
$u|_{t\in [-\delta,\delta]}$ and $v|_{t\in [-\delta,\delta]}$ such that $\|Iu\|_{X_{0,b}^\delta}=
\|I\tilde{u}\|_{X_{0,b}}$ and $\|Iv\|_{X_{0,b}^\delta}=
\|I\tilde{v}\|_{X_{0,b}}$.
\end{prop}

\noindent {\it Remark.} As we described in Section 1, the global
well-posedness result shall be effected by the estimate
(\ref{delta}) on the lifetime. By the equivalence (\ref{II}), we
have from Proposition 3.1 that $\delta^{\frac{1}{2}-}\sim
 \|Iu_0\|_{L^2}^{-1}$. On the other
hand, a similar local well-posedness result also can be achieved by a
standard process. In fact, we can obtain the bilinear estimates
(better than what obtained in \cite{KPV1}) that
$$
\|\partial_x(uv)\|_{X_{s,b-1}}\lesssim \|u\|_{X_{s,b'}}
\|u\|_{X_{s,b'}}
$$
for any $s\in (-\dfrac{3}{4},0), b\in
(\dfrac{1}{2},\dfrac{3}{4}+\dfrac{1}{3}s], b'\in (\dfrac{1}{2},1]$
(we omit the proofs here). Then by a general result (see
\cite{CKSTT3}), we have
$$
\|\partial_xI(uv)\|_{X_{0,b-1}}\lesssim \|Iu\|_{X_{0,b'}}
\|Iu\|_{X_{0,b'}}
$$
under the same assumptions. Thus we can establish the local
well-posedness similar to Proposition 3.1 but replacing the lifetime
estimate by
$$
\delta^{b-b'}\sim\|Iu_0\|_{L^2}^{-1}.
$$
However, since $b-b'<\dfrac{1}{4}+\dfrac{1}{3}s$, it is much weaker than
(\ref{delta}).

Now we turn to an arithmetic fact which is often used below, before
the proof of the main results of this section. We note that
\begin{equation}
|(\tau-\phi(\xi))-(\tau_1-\phi(\xi_1))-(\tau_2-\phi(\xi_2))|
\gtrsim  |\xi||\xi_1||\xi_2|,\label{claim}
\end{equation}
where $\xi=\xi_1+\xi_2$ and $\max\{|\xi|,|\xi_1|,|\xi_2|\}\geq a$.
In fact, we may assume that $|\xi_1|\geq |\xi_2|$ by symmetry, then
$\xi\cdot\xi_1\geq 0$. We only consider the case:
$\xi,\xi_1\geq 0, \xi_2\leq 0$
(the other three cases can be verified similarly), then
\beqs
\begin{split}
 & \ (\tau-\phi(\xi))-(\tau_1-\phi(\xi_1))-(\tau_2-\phi(\xi_2))
 = \phi(\xi_1)+\phi(\xi_2)-\phi(\xi)\\
 = \ & -\nu(\xi_1^2-\xi_2^2-\xi^2)-3\mu\xi\xi_1\xi_2
 = \xi\xi_2(2\nu-3\mu\xi_1),
\end{split}
\eeqs
since $|\xi_1|\geq a$, we have (\ref{claim}). According to (\ref{claim}),
one of the following three cases always occurs:
\begin{equation}\label{abc}
(a)\,|\tau-\phi(\xi)|\gtrsim |\xi||\xi_1||\xi_2|;\ \
(b)\,|\tau_1-\phi(\xi_1)|\gtrsim |\xi||\xi_1||\xi_2|;\ \
(c)\,|\tau_2-\phi(\xi_2)|\gtrsim |\xi||\xi_1||\xi_2|.
\end{equation}

\noindent{\it Proof of Proposition 3.2.} By duality and Plancherel's
identity, it suffices to show that \beq
 \displaystyle\int\!\!\!\!\int\xi
 m(\xi)\widehat{h_\delta}(\xi,\tau)\widehat{\tilde{u}\tilde{v}}(\xi,\tau)\,d\xi
 d\tau \lesssim
 K\|h\|_{X_{0,1-b}}\>
 \|I\tilde{u}\|_{X_{0,b}}\>
 \|I\tilde{v}\|_{X_{0,b}}
 \equiv K \cdot  RHS
 \label{Thm1}
\eeq for any $h\in X_{0,1-b}$. Here we denote
$h_\delta(x,t)=\psi(t/\delta) h(x,t)$ and
$K=\delta^{\frac{1}{2}-}+N^{-\frac{3}{2}+}$ for short.

We write
$$
\hat{f}(\xi,\tau) =\widehat{I\tilde{u}}(\xi,\tau)
=m(\xi)\widehat{\tilde{u}}(\xi,\tau),\quad
\hat{g}(\xi,\tau) =\widehat{I\tilde{v}}(\xi,\tau)
=m(\xi)\widehat{\tilde{v}}(\xi,\tau),
$$
then (\ref{Thm1}) is
changed into
\begin{eqnarray}
&
LHS &\equiv
     \displaystyle\int_{\ast}|\xi|\dfrac{m(\xi)}{m(\xi_1)m(\xi_2)}
     \widehat{h_\delta}(\xi,\tau)\>
     \hat{f}(\xi_1,\tau_1)\>
     \hat{g}(\xi_2,\tau_2)\nonumber\\
    &&\lesssim
     K\|h\|_{X_{0,1-b}}\>
     \|f\|_{X_{0,b}}\>
     \|g\|_{X_{0,b}}
     =K\cdot RHS,
     \label{3.5}
\end{eqnarray}
where $\displaystyle\int_\ast
=\int_{\stackrel{\xi_1+\xi_2=\xi,}{ \tau_1+\tau_2=\tau}}\,d\xi_1
d\xi_2 d\tau_1 d\tau_2$, which is corresponding to convolution.

Without loss of generality, we may assume further
$\widehat{h_\delta},\hat{f},\hat{g}\in
L^2(\R^2)$ are nonnegative functions. By symmetry, we may consider
only the integration over the region of $|\xi_1|\geq|\xi_2|$ (so,
$\xi\cdot\xi_1\geq 0$ and $|\xi|\leq 2|\xi_1|$), and divide it into
the following different parts.
$$
\begin{array}{ll}
\mbox{Part 1. } |\xi|,|\xi_1|,|\xi_2|\lesssim N;&
\mbox{Part 2. } |\xi|\lesssim N, |\xi_1|,|\xi_2|\gg N;\\
\mbox{Part 3. } |\xi_2|\lesssim N, |\xi|,|\xi_1|\gg N;\qquad\qquad&
\mbox{Part 4. } |\xi|,|\xi_1|,|\xi_2|\gg N.
\end{array}
$$

\noindent {\bf Part 1}. $|\xi|,|\xi_1|,|\xi_2|\lesssim N$.

In this part, $m(\xi),m(\xi_1),m(\xi_2)\in [C^s,1]$ for some
constant $C$, therefore
$$
 LHS\sim
 \displaystyle\int_{\ast}
|\xi| \>\widehat{h_\delta}(\xi,\tau) \>\hat{f}(\xi_1,\tau_1)
\>\hat{g}(\xi_2,\tau_2) .
$$
Here and below $LHS$ denotes the integral in the \emph{left hand
side} of (\ref{3.5}) over the corresponding part of the integration
region.

{\bf Subpart (I)}. $|\xi|\lesssim a$.\quad By Plancherel's identity,
H\"{o}lder's inequality, (\ref{XE3}), (\ref{cor}) and Lemma 2.7(iv), we have
\beqs\begin{split} LHS
    &\lesssim
      \displaystyle\int_{\ast}
      \widehat{h_\delta}(\xi,\tau) \>\hat{f}(\xi_1,\tau_1)
      \>\hat{g}(\xi_2,\tau_2)\\
    &=\displaystyle\int
      h_\delta(x,t)\> f(x,t)\> g(x,t)\,dxdt\\
    &\leq
      \|h_\delta\|_{L^2_{xt}}\>
      \|f\|_{L^{4}_{xt}}\>
      \|g\|_{L^4_{xt}}\\
   &\lesssim
      \|h_\delta\|_{X_{0,0}}\>
      \|I\tilde{u}\|_{X_{0,\frac{1}{3}+}}\>
      \|I\tilde{v}\|_{X_{0,\frac{1}{3}+}}\\
    &\lesssim \delta^{\frac{5}{6}-} RHS.
\end{split}
\eeqs

{\bf Subpart (II)}. $|\xi_1|\gg |\xi_2|$ and $|\xi|\gg a$.\quad In
this subpart we have $|\xi|\sim |\xi_1|$, and
$$
|\phi'(\xi_1)-\phi'(\xi_2)|\sim |\xi|^2.
$$
Therefore, by Plancherel's identity,
H\"{o}lder's inequality, (\ref{IE}) and Lemma 2.7(iv), we have
\beqs
\begin{split}
LHS &\lesssim
      \displaystyle\int
      \widehat{h_\delta}(\xi,\tau)\>
      \widehat{I^{\frac{1}{2}}(f,g)}(\xi,\tau)\,d \xi d \tau\\
    &\leq
      \|h_\delta\|_{L^2_{xt}}\>
      \left\|I^{\frac{1}{2}}(f,g)\right\|_{L^2_{xt}}\\
    &\lesssim
      \|h_\delta\|_{X_{0,0}}\>\|f\|_{X_{0,b}}\>\|g\|_{X_{0,b}}\\
    &\lesssim \delta^{\frac{1}{2}-} RHS,
\end{split}
\eeqs

{\bf Subpart (III)}. $|\xi_1|\sim |\xi_2|$ and $|\xi|\gg a$.
We split the integration into three cases according to (\ref{abc}).

(a) $|\tau-\phi(\xi)|\gtrsim |\xi||\xi_1||\xi_2|\gtrsim |\xi|^3$ (since
$|\xi|\lesssim |\xi_1|$). Then by
(\ref{XE3}), Lemma 2.7(iv) and (\ref{cor}),
\beqs
\begin{split}
 LHS &\lesssim
       \displaystyle\int_{\ast}
       \langle\tau-\phi(\xi)\rangle^{\frac{1}{3}}
       \widehat{h_\delta}(\xi,\tau) \>\hat{f}(\xi_1,\tau_1)
       \>\hat{g}(\xi_2,\tau_2)\\
     &\leq
       \|h_\delta\|_{X_{0,\frac{1}{3}}}\>
       \|f\|_{L^4_{xt}}\>
       \|g\|_{L^4_{xt}}\\
     &\lesssim
       \|h_\delta\|_{X_{0,\frac{1}{3}}}\>
       \|f\|_{X_{0,\frac{1}{3}+}}\>
       \|g\|_{X_{0,\frac{1}{3}+}}\\
     &\lesssim \delta^{\frac{1}{2}-} RHS.
\end{split}
\eeqs

For cases (b) and (c), the estimation is similar to (a), and we omit the details.

\noindent {\bf Part 2}. $|\xi|\lesssim N$, $|\xi_1|,|\xi_2|\gg N$.

In this part, $|\xi_1|\sim|\xi_2|\gg |\xi|$. By (\ref{m}),
$$
LHS\lesssim N^{2s} \displaystyle\int_{\ast}
|\xi||\xi_2|^{-2s}\widehat{h_\delta}(\xi,\tau) \>\hat{f}(\xi_1,\tau_1)
\>\hat{g}(\xi_2,\tau_2).\\
$$
We split the integration into three cases according to (\ref{abc}).

(a)  $|\tau-\phi(\xi)|\gtrsim |\xi||\xi_1||\xi_2|$. Since
$$
|\phi'(\xi_1)-\phi'(\xi_2)|\sim |\xi||\xi_1|,
$$
then by the definition of (\ref{Is}) and note that $s>-\dfrac{3}{4}$, we have
\begin{eqnarray}
 &LHS &\lesssim
       N^{2s}\displaystyle\int_{\ast}
       |\xi|^b|\xi_2|^{-2s+2b-2}\>\langle\tau-\phi(\xi)\rangle^{1-b}
       \widehat{h_\delta}(\xi,\tau) \>\hat{f}(\xi_1,\tau_1)
       \>\hat{g}(\xi_2,\tau_2)\nonumber\\
     &&\lesssim
       N^{2s}\displaystyle\int_{\ast}
       |\xi|^{b-\frac{1}{2}}|\xi_2|^{-2s+2b-\frac{5}{2}}
       \>\langle\tau-\phi(\xi)\rangle^{1-b}
       \widehat{h_\delta}(\xi,\tau)
       \>|\xi|^{\frac{1}{2}}|\xi_1|^{\frac{1}{2}}\hat{f}(\xi_1,\tau_1)
       \>\hat{g}(\xi_2,\tau_2)\nonumber\\
     &&\lesssim
       N^{3b-3}\displaystyle\int
       \>\langle\tau-\phi(\xi)\rangle^{1-b}
       \widehat{h_\delta}(\xi,\tau)\>
       \widehat{I^{\frac{1}{2}}(f,g)}(\xi,\tau)\,d \xi d \tau,\label{3.8}
\end{eqnarray}
since $-2s+2b-\dfrac{5}{2}\leq 0$ by choosing
$b-\dfrac{1}{2}$ small enough. Therefore, by (\ref{IE}) and Lemma 2.7 (iv),
(\ref{3.8}) is controlled by
$$
       N^{-\frac{3}{2}+}
       \|h_\delta\|_{X_{0,1-b}}
       \left\|I^{\frac{1}{2}}(f,g)\right\|_{L^2_{xt}}\\
       \lesssim
       N^{-\frac{3}{2}+}RHS.
$$

(b)  $|\tau_1-\phi(\xi_1)|\gtrsim |\xi||\xi_1||\xi_2|$. Since
$$
|\phi'(\xi)-\phi'(\xi_2)|\sim |\xi_2|^2
$$
in this situation, then by the definition of (\ref{Is1}), we have
\begin{eqnarray}
 &LHS &\lesssim
       N^{2s}\displaystyle\int_{\ast}
       |\xi|^{1-b}|\xi_2|^{-2s-2b}\>
       \widehat{h_\delta}(\xi,\tau)
       \>\langle\tau_1-\phi(\xi_1)\rangle^{b}\hat{f}(\xi_1,\tau_1)
       \>\hat{g}(\xi_2,\tau_2)\nonumber\\
     &&\lesssim
       N^{2s}\displaystyle\int_{\ast}
       |\xi|^{1-b}|\xi_2|^{-2s-2b-2\theta}
       \langle\tau_1-\phi(\xi_1)\rangle^{b}\hat{f}(\xi_1,\tau_1)
       \>
       |\xi_2|^{2\theta}\widehat{h_\delta}(\xi,\tau)\hat{g}(\xi_2,\tau_2),
       \label{3.9}
\end{eqnarray}
where $\theta=(\dfrac{5}{4}-\dfrac{3}{2}b)-$ such that
$1-b>\dfrac{1}{6}+\dfrac{3}{2}\theta$ and $-s-b-\theta<0$. Therefore,
by (\ref{IE1}) in the version of (\ref{Is1}) and Lemma 2.7 (iv),
(\ref{3.9}) is controlled by
\beqs
\begin{split}
     &N^{1-3b-2\theta}\displaystyle\int
       \langle\tau_1-\phi(\xi_1)\rangle^{b}\hat{f}(\xi_1,\tau_1)\>
       \widehat{I^{\theta}(h_\delta,g)}(\xi_1,\tau_1)
       \,d \xi_1 d \tau_1\\
     \lesssim &
       N^{-\frac{3}{2}+}
       \|f\|_{X_{0,b}}\>
       \left\|I^{\theta}(h_\delta,g)\right\|_{L^2_{xt}}\\
     \lesssim &
       N^{-\frac{3}{2}+}
       \|f\|_{X_{0,b}}\>
       \|h_\delta\|_{X_{0,1-b}}\>
       \|g\|_{X_{0,b}}\\
     \lesssim &
       N^{-\frac{3}{2}+}RHS,
\end{split}
\eeqs
where we note that $1-3b-2\theta=-\dfrac{3}{2}+$.

The part (c) is similar to (b), and the details are omitted.

\noindent {\bf Part 3}. $|\xi_2|\lesssim N$, $|\xi|,|\xi_1|\gg N$.

In this part,
$$
 LHS\sim
 \displaystyle\int_{\ast}|\xi|\>
 \widehat{h_\delta}(\xi,\tau)\>\hat{f}(\xi_1,\tau_1)
 \>\hat{g}(\xi_2,\tau_2).
$$
The estimation in this part is the same as Part 1 (II), and $LHS$ can
be controlled by $\delta^{\frac{1}{2}-} RHS.$

\noindent {\bf Part 4}. $|\xi|,|\xi_1|,|\xi_2|\gg N$.

In this part,
$$
LHS\lesssim N^{s} \displaystyle\int_{\ast}
|\xi|^{1+s}|\xi_1|^{-s}|\xi_2|^{-s}\widehat{h_\delta}(\xi,\tau)
\>\hat{f}(\xi_1,\tau_1)
\>\hat{g}(\xi_2,\tau_2).\\
$$
First, we divide the integral into two subparts, then in each subpart below
we split it again into three subsubparts by (\ref{abc}).

 {\bf Subpart (I)}. $|\xi|\ll|\xi_1|$, then $|\xi_1|\sim|\xi_2|$.

(a)  $|\tau-\phi(\xi)|\gtrsim |\xi||\xi_1||\xi_2|\sim |\xi||\xi_2|^2$. Since
$$
|\phi'(\xi_1)-\phi'(\xi_2)|\sim |\xi||\xi_1|,
$$
then by (\ref{IE}) and Lemma 2.7 (iv), we have
\begin{eqnarray*}
 &LHS &\lesssim
       N^{s}\displaystyle\int_{\ast}
       |\xi|^{s+b}|\xi_2|^{-2s+2b-2}\>\langle\tau-\phi(\xi)\rangle^{1-b}
       \widehat{h_\delta}(\xi,\tau) \>\hat{f}(\xi_1,\tau_1)
       \>\hat{g}(\xi_2,\tau_2)\nonumber\\
     &&\leq
       N^{s}\displaystyle\int_{\ast}
       |\xi|^{s+b-\frac{1}{2}}|\xi_2|^{-2s+2b-\frac{5}{2}}
       \>\langle\tau-\phi(\xi)\rangle^{1-b}
       \widehat{h_\delta}(\xi,\tau)
       \>|\xi|^{\frac{1}{2}}|\xi_1|^{\frac{1}{2}}\hat{f}(\xi_1,\tau_1)
       \>\hat{g}(\xi_2,\tau_2)\\
     &&\lesssim
       N^{3b-3}\displaystyle\int
       \>\langle\tau-\phi(\xi)\rangle^{1-b}
       \widehat{h_\delta}(\xi,\tau)\>
       \widehat{I^{\frac{1}{2}}(f,g)}(\xi,\tau)\,d \xi d \tau\\
     &&\lesssim
       N^{-\frac{3}{2}+}
       \|h_\delta\|_{X_{0,1-b}}
      \left\|I^{\frac{1}{2}}(f,g)\right\|_{L^2_{xt}}\\
     &&\lesssim
       N^{-\frac{3}{2}+}RHS,
\end{eqnarray*}
since $s+b-\dfrac{1}{2}\leq 0$ and $-2s+2b-\dfrac{5}{2}\leq 0$ in the third step.

(b)  $|\tau_1-\phi(\xi_1)|\gtrsim |\xi||\xi_1||\xi_2|\sim |\xi||\xi_2|^2$. Since
$$
|\phi'(\xi)-\phi'(\xi_2)|\sim |\xi_2|^2
$$
in this situation, then by (\ref{IE1}) in the version of (\ref{Is1})
and Lemma 2.7 (iv), we have
\begin{eqnarray*}
 &LHS &\lesssim
       N^{s}\displaystyle\int_{\ast}
       |\xi|^{1+s-b}|\xi_2|^{-2s-2b}\>
       \widehat{h_\delta}(\xi,\tau)
       \>\langle\tau_1-\phi(\xi_1)\rangle^{b}\hat{f}(\xi_1,\tau_1)
       \>\hat{g}(\xi_2,\tau_2)\\
     &&\lesssim
       N^{1-3b-2\theta}\displaystyle\int
       \langle\tau_1-\phi(\xi_1)\rangle^{b}\hat{f}(\xi_1,\tau_1)\>
       \widehat{I^{\theta}(h_\delta,g)}(\xi_1,\tau_1)
       \,d \xi_1 d \tau_1\\
     &&\lesssim
       N^{-\frac{3}{2}+}
       \|f\|_{X_{0,b}}\>
       \left\|I^{\theta}(h_\delta,g)\right\|_{L^2_{xt}}\\
     &&\lesssim
       N^{-\frac{3}{2}+}
       \|f\|_{X_{0,b}}\>
       \|h_\delta\|_{X_{0,1-b}}\>
       \|g\|_{X_{0,b}}\\
     &&\lesssim
       N^{-\frac{3}{2}+}RHS.
\end{eqnarray*}

The part (c) is got in the same way as (b), so we omit the details
again.

{\bf Subpart (II)}. $|\xi|\sim|\xi_1|$.

(a) $|\tau-\phi(\xi)|\gtrsim |\xi||\xi_1||\xi_2|\sim |\xi|^2|\xi_2|$.
By Lemma 2.3 we have
\beqs
\begin{split}
 LHS &\lesssim N^{s}
       \displaystyle\int_{\ast}
       |\xi|^{2b-1}|\xi_2|^{-s+b-1}
       \>\langle\tau-\phi(\xi)\rangle^{1-b}
       \widehat{h_\delta}(\xi,\tau) \>\hat{f}(\xi_1,\tau_1)
       \>\hat{g}(\xi_2,\tau_2)\\
     &\lesssim N^{s}
       \displaystyle\int_{\ast}
       |\xi|^{2b-\frac{9}{8}}|\xi_2|^{-s+b-\frac{9}{8}}
       \>\langle\tau-\phi(\xi)\rangle^{1-b}
       \widehat{h_\delta}(\xi,\tau)
       \>\widehat{D_x^{\frac{1}{8}}P^af}(\xi_1,\tau_1)
       \>\widehat{D_x^{\frac{1}{8}}P^ag}(\xi_2,\tau_2)\\
     &\lesssim N^{3b-\frac{9}{4}}
       \|h_\delta\|_{X_{0,1-b}}
       \left\|D_x^{\frac{1}{8}}P^af\right\|_{L^4_{xt}}\>
       \left\|D_x^{\frac{1}{8}}P^ag\right\|_{L^4_{xt}}\\
     &\lesssim N^{3b-\frac{9}{4}}
       \|h_\delta\|_{X_{0,1-b}}\>
       \|f\|_{X_{0,\frac{3}{8}+}}\>
       \|g\|_{X_{0,\frac{3}{8}+}}\\
     &\lesssim N^{-\frac{3}{4}+}\delta^{\frac{1}{4}-} RHS,
\end{split}
\eeqs
since $|\xi_2|\lesssim |\xi|$ and $3b-s-\dfrac{9}{4}\leq 0$ in the third step.

(b) $|\tau_1-\phi(\xi_1)|\gtrsim |\xi||\xi_1||\xi_2|\sim |\xi|^2|\xi_2|$.
By Lemma 2.3 we have
\beqs
\begin{split}
 LHS &\lesssim N^{s}
       \displaystyle\int_{\ast}
       |\xi|^{1-2b}|\xi_2|^{-s-b}
       \>\widehat{h_\delta}(\xi,\tau)
       \>\langle\tau_1-\phi(\xi_1)\rangle^{b}\hat{f}(\xi_1,\tau_1)
       \>\hat{g}(\xi_2,\tau_2)\\
     &\lesssim N^{s}
       \displaystyle\int_{\ast}
       |\xi|^{\frac{7}{8}-2b}|\xi_2|^{-s-b-\frac{1}{8}}
       \>\widehat{D_x^{\frac{1}{8}}P^ah_\delta}(\xi,\tau)
       \>\langle\tau_1-\phi(\xi_1)\rangle^{b}\widehat{f}(\xi_1,\tau_1)
       \>\widehat{D_x^{\frac{1}{8}}P^ag}(\xi_2,\tau_2)\\
     &\lesssim N^{-3b+\frac{3}{4}}
       \left\|D_x^{\frac{1}{8}}P^ah_\delta\right\|_{L^4_{xt}}
       \|f\|_{X_{0,b}}\>
       \left\|D_x^{\frac{1}{8}}P^ag\right\|_{L^4_{xt}}\\
     &\lesssim N^{-3b+\frac{3}{4}}
       \|h_\delta\|_{X_{0,\frac{3}{8}+}}\>
       \|f\|_{X_{0,b}}\>
       \|g\|_{X_{0,\frac{3}{8}+}}\\
     &\lesssim N^{-\frac{3}{4}-}\delta^{\frac{1}{4}-} RHS,
\end{split}
\eeqs

The part (c) is treated in a very similar manner as (a) and (b), so
we omit the details. This completes the proof of the proposition.
\hfill$\Box$

Now, we are ready to prove Proposition 3.1. For this
purpose, we define the operator $\Phi_\delta$ as
\beq
 \Phi_\delta u(t)
 =\psi(t)S(t)u_0+\psi(t)
 \displaystyle\int_0^tS(t-s)
 \psi(s/\delta)\partial_x\tilde{u}^2(s)
 \,ds,
 \label{Phi}
\eeq
where $\tilde{u}$ is the extension of
$u|_{t\in [-\delta,\delta]}$ such that $\|Iu\|_{X_{0,b}^\delta}=
\|I\tilde{u}\|_{X_{0,b}}$ ($I=I_{N,s}$). Then (\ref{kbo})-(\ref{1.2})
is locally well-posed if
$\Phi_\delta$ has a unique fixed point.

Acting the operator $I$ onto both sides of (\ref{Phi}),
taking the $X_{0,b}^\delta$-norm for $b=\dfrac{1}{2}+$,
and employing Lemma 2.7 and (\ref{Thm}) we have
\beqs
\begin{split}
  \|I(\Phi_\delta u)\|_{X_{0,b}^\delta}
&\lesssim
  \|\psi(t)S(t)Iu_0\|_{X_{0,b}^\delta}
  +\left\|\psi(t)\displaystyle\int_0^tS(t-s)
   \psi(s/\delta)\partial_xI(\tilde{u}^2(s))\,ds
   \right\|_{X_{0,b}^\delta}\\
&\lesssim
  \|Iu_0\|_{L^2}
  +\left\|\psi(t/\delta)\partial_xI(\tilde{u}^2)
  \right\|_{X_{0,b-1}}\\
&\leq
  C_1\|Iu_0\|_{L^2}+C_2(\delta^{\frac{1}{2}-}
  +N^{-\frac{3}{2}-})\|Iu\|_{X_{0,b}^\delta}^2.
\end{split}
\eeqs

Consider
$$
B_r=\{u:Iu\in X_{0,b}^\delta, \makebox{\,\,such that\,\,}
\|Iu\|_{X_{0,b}^\delta}\leq r\},
$$
where $r=2C_1\|Iu_0\|_{L^2}$, some small $\delta$ and large $N$ will be
decided later. Then $B_r$ is a complete metric space.
Observing that if we choose $N$, $\delta$ such that
\beq
 C_2(\delta^{\frac{1}{2}-}+N^{-\frac{3}{2}-})r\leq
 \dfrac{1}{2},\label{Nd1}
\eeq
then the operator $\Phi_\delta$ maps $B_r$ into itself.
(\ref{Nd1}) is valid if we choose $N,\delta$ such that
\beq
 100C_1C_2N^{-\frac{3}{2}-}\cdot \|Iu_0\|_{L^2}\leq 1
 \makebox{\quad and \quad } 100C_1C_2\delta^{\frac{1}{2}-}\leq
 \|Iu_0\|_{L^2}^{-1}.\label{Nd}
\eeq

Similarly, under the condition (\ref{Nd}), one has
$$
\|I(\Phi_\delta(u)-\Phi_\delta(v))\|_{X_{0,b}^\delta}\leq
\dfrac{1}{2}\|I(u-v)\|_{X_{0,b}^\delta},\quad \forall \ u,v\in B_r,
$$
and thus $\Phi_\delta$ is a contraction on
$B_r$. Thus by the fixed point
theorem, we complete the proof of Proposition 3.1.

 \vspace{0.5cm}
\section{The Global Well-posedness}

In this section, we consider the global well-posedness of
(\ref{kbo})-(\ref{1.2}) by adopting the argument in
\cite{CKSTT}, which is based on a
multilinear correction technique and iteration. Unfortunately, the
phase function $\phi(\xi)$ loses some symmetries which brings much
convenience in \cite{CKSTT} to obtain some pointwise estimates,
this makes many difficulties.
To overcome these difficulties, we use some multiplier decomposition argument
to deal with it. More precisely, we also introduce the third version
modified energy, but we only use it to cancel a part of ``correction term''
in the second modified energy. Similar argument were appeared
previously in \cite{BDGS}, \cite{Bourgain2}, \cite{CKSTT-08}.

\subsection{Modified Energies and I-method}

First we  observe some arithmetic facts. Recall that
$\phi(\xi)=-\nu\xi|\xi|+\mu\xi^3$, let
$$
\alpha_k\equiv i(\phi(\xi_1)+\cdots+\phi(\xi_k)).
$$
Then $\alpha_2=0$ for $\xi_1+\xi_2=0$. Similar to (\ref{claim}),
when $\xi_1+\xi_2+\xi_3=0$ with $\max\{|\xi_1|,|\xi_2|,|\xi_3|\}\geq a$,
then
\beq
 |\alpha_3|\gtrsim|\xi_1||\xi_2||\xi_3|.
 \label{fact2}
\eeq

Next, we state the definitions of the modified energies
and adopt the notations in \cite{CKSTT}.

In this
section, let $u$ be the real-valued solution of (\ref{kbo})-(\ref{1.2}).
For a given function $m(\xi_1,\cdots,\xi_k)$ defined on the hyperplane
$$
\Gamma_k = \{(\xi_1,\cdots,\xi_k):\xi_1+\cdots+\xi_k=0\},
$$
we define
$$
  \Lambda_k(m)
 =\displaystyle\int_{\Gamma_k}m(\xi_1,\cdots,\xi_k)
   \prod_{j=1}^k\F_xu(\xi_j,t)\,d\xi_1\cdots d\xi_{k-1}.
$$
Denote the modified energy as
$$
E^2_I(t)\equiv\|Iu(t)\|_{L^2}^2=\Lambda_2(m(\xi_1)m(\xi_2)),
$$
then by the arithmetic fact above, (\ref{kbo}) and a direct computation
(cf. \cite{CKSTT}), one has
$$
\dfrac{d}{dt}E^2_I(t)
=\Lambda_2(m(\xi_1)m(\xi_2)\alpha_2)+\Lambda_3(M_3)
=\Lambda_3(M_3),
$$
where
\begin{equation}
  M_3(\xi_1,\xi_2,\xi_3)
 =-\frac{2i}{3}\big(m^2(\xi_1)\xi_1+m^2(\xi_2)\xi_2+m^2(\xi_3)\xi_3\big).\label{M3}
\end{equation}
Define the second modified energy $E^3_I(t)$ by
\beq
E^3_I(t)=\Lambda_3(\sigma_3)+E^2_I(t),
\label{E3}
\eeq
where
$$
  \sigma_3(\xi_1,\xi_2,\xi_3)
 =-M_3(\xi_1,\xi_2,\xi_3)\big/\alpha_3(\xi_1,\xi_2,\xi_3),
$$
then one has
\beq
\dfrac{d}{dt}E^3_I(t)=\Lambda_4(M_4).\label{E3'}
\eeq
where
\begin{equation}\label{M4}
M_4(\xi_1,\cdots,\xi_4)
 =-\dfrac{3}{4}i[\sigma_3(\xi_1,\xi_2,\xi_3+\xi_4)(\xi_3+\xi_4)]_{sym}.
\end{equation}

We denote the sets that
$$
  \Omega = \{(\xi_1,\xi_2,\xi_3,\xi_4)\in \Gamma_4:
  |\xi_1|, \cdots, |\xi_4|\gtrsim N\},
$$
and rewrite (\ref{E3'}) by
\beq
\dfrac{d}{dt}E^3_I(t)=\Lambda_4(\bar{M}_4)+\Lambda_4(\tilde{M}_4),\label{E3''}
\eeq
where
\begin{equation}\label{M4-1}
    \bar{M}_4=(\chi_{\Gamma_4}-\chi_\Omega)M_4;\quad
    \tilde{M}_4=\chi_\Omega M_4.
\end{equation}
Now, we define the third modified energy $E^4_I(t)$ as
\beq
E^4_I(t)=\Lambda_4(\sigma_4)+E^3_I(t),
\label{E4}
\eeq
where
$$
  \sigma_4(\xi_1,\cdots,\xi_4)
 =-\frac{\tilde{M}_4(\xi_1,\cdots,\xi_4)}{\alpha_4(\xi_1,\cdots,\xi_4)}.\nonumber\\
$$
Then one has
\beq
\dfrac{d}{dt}E^4_I(t)=\Lambda_4(\bar{M}_4)+\Lambda_5(M_5).\label{E4'}
\eeq
where
\begin{equation}\label{M5}
      M_5(\xi_1,\cdots,\xi_5)
 =-\dfrac{4}{5}i[\sigma_4(\xi_1,\xi_2,\xi_3,\xi_4+\xi_5)(\xi_4+\xi_5)]_{sym}.
\end{equation}

\noindent {\it Remark.} The second version of modified energy is not
enough to obtain the claim result in Theorem 1.1 for $s>\dfrac{3}{4}$.
Indeed, the best estimate (in our opinion) on almost conserved quantity is
$$
 \sup\limits_{t\in [0,\delta]}\|Iu(t)\|_{L^2}^2
 \lesssim
 \|Iu_0\|_{L^2}^2
 +(N^{-\frac{3}{2}+}\delta^{\frac{1}{2}-}+N^{-3+})
  \|Iu_0\|_{L^2}^4,
$$
which implies that (\ref{kbo})-(\ref{1.2}) is globally well-posed in
$H^s(\R)$ when $s>-1/2$. So one may need to introduce the third
modified energy. For this purpose, as the general argument, one may
defined $\sigma_4 =-\dfrac{{M}_4}{\alpha_4}$, and give the definition
of $M_5$ as (\ref{M5}). Unfortunately, it is much hard to give the
pointwise estimates on this $M_5$, especially because of the complexity of
the phase function $\phi(\xi)$. We note that the most hand case appears
when the third and the fourth highest values of $|\xi_1|,\cdots, |\xi_4|$:
$|\xi_3^*|,|\xi_4^*|\ll N$,
but on the other hand, this case behaves well in the estimation of the increment
of $E^3_I(t)$. This is the reason that we
divided the multiplier $M_4$ into two parts and use the multiplier decomposition
argument to deal with it.
The argument brings much convenient for us in this paper.

\subsection{Pointwise Multiplier Bounds}

By mean value theorem, one has an estimate on $M_3$:
\begin{lem}
Let $|\xi_{min}|=\min\{|\xi_1|,|\xi_2|,|\xi_3|\}, \xi_1+\xi_2+\xi_3=0 $,
then,
\beq
|M_3(\xi_1,\xi_2,\xi_3)|\lesssim
m^2(\xi_{min})|\xi_{min}|,\label{EM3}
\eeq

\end{lem}

Now we state some simple facts.

\begin{lem} The following estimates hold,
\bi

\item[{\rm(i)}]
           $\alpha_3(\xi_1,\xi_2,\xi_3+\xi_4)+\alpha_3(\xi_3,\xi_4,\xi_1+\xi_2)
           =\alpha_4(\xi_1,\xi_2,\xi_3,\xi_4)$;

\item[{\rm(ii)}]
          $|\alpha_4(\xi_1,\xi_2,\xi_3,\xi_4)|
          \sim |\xi_1+\xi_2||\xi_1+\xi_3||\xi_1+\xi_4|$;

\item[{\rm(iii)}]
          $|m^2(\xi_1)\xi_1+m^2(\xi_2)\xi_2
          +m^2(\xi_3)\xi_3+m^2(\xi_4)\xi_4|\lesssim
          |\alpha_4(\xi_1,\xi_2,\xi_3,\xi_4)|/|\xi_m|^2$.
\ei
\end{lem}
\begin{proof}
(i) easily follows from a direct check.
For (ii), we may assume that $|\xi_1|\geq |\xi_2|\geq|\xi_3|\geq|\xi_4|$
by symmetry. By the facts that
$$
\alpha_4(\xi_1,\xi_2,\xi_3,\xi_4)
=-\nu(\xi_1|\xi_1|+\xi_2|\xi_2|+\xi_3|\xi_3|+\xi_4|\xi_4|)+
\mu(\xi_1^3+\xi_2^3+\xi_3^3+\xi_4^3)
$$
and
$$
\xi_1^3+\xi_2^3+\xi_3^3+\xi_4^3=(\xi_1+\xi_2)(\xi_1+\xi_3)(\xi_1+\xi_4),
$$
we only need to show
$$
\xi_1|\xi_1|+\xi_2|\xi_2|+\xi_3|\xi_3|+\xi_4|\xi_4|
\ll|\xi_1+\xi_2||\xi_1+\xi_3||\xi_1+\xi_4|.
$$
For this purpose, we may assume that $\xi_1>0$ and
split into three cases as following:
\begin{eqnarray}
&(1), \xi_1>0, \xi_2>0,\xi_3<0,\xi_4<0;\quad
(2), \xi_1>0, \xi_2<0,\xi_3<0,\xi_4>0;\nonumber\\
&(3), \xi_1>0, \xi_2<0,\xi_3<0,\xi_4<0.\label{4.71}
\end{eqnarray}
For (1),
\begin{eqnarray*}
&&\big|\xi_1|\xi_1|+\xi_2|\xi_2|+\xi_3|\xi_3|+\xi_4|\xi_4|\big|
=|\xi_1^2+\xi_2^2-\xi_3^2-\xi_4^2|\\
&=&|\xi_1+\xi_3||\xi_1+\xi_4|
\ll|\xi_1+\xi_2||\xi_1+\xi_3||\xi_1+\xi_4|.
\end{eqnarray*}
For (2),
$$
\big|\xi_1|\xi_1|+\xi_2|\xi_2|+\xi_3|\xi_3|+\xi_4|\xi_4|\big|
=|\xi_1+\xi_2||\xi_1+\xi_3|
\ll|\xi_1+\xi_2||\xi_1+\xi_3||\xi_1+\xi_4|.
$$
For (3), on one hand,
\begin{eqnarray*}
&&\big|\xi_1|\xi_1|+\xi_2|\xi_2|+\xi_3|\xi_3|+\xi_4|\xi_4|\big|
=|\xi_1^2-\xi_2^2-\xi_3^2-\xi_4^2|\\
&\leq&|\xi_1+\xi_2||\xi_1-\xi_2|+|\xi_3+\xi_4|^2
\sim|\xi_1+\xi_2||\xi_1|;
\end{eqnarray*}
on the other hand, we note that $|\xi_1|-|\xi_3|\gtrsim |\xi_1|,$ so
$$
|\xi_1+\xi_2||\xi_1+\xi_3||\xi_1+\xi_4|\sim|\xi_1+\xi_2||\xi_1|^2.
$$
Thus we have the claim (ii).

For (iii), we also assume that $|\xi_1|\geq |\xi_2|\geq|\xi_3|\geq|\xi_4|$
and split it into three cases as (\ref{4.71}). Then, for (1), we have
$|\xi_1|\sim|\xi_4|$, thus by the double mean value theorem
(see \cite{CKSTT}, for example),
\begin{eqnarray*}
|m^2(\xi_1)\xi_1+\cdots+m^2(\xi_4)\xi_4|
&\lesssim&
|\xi_1+\xi_3||\xi_1+\xi_4||f''(\xi_1)|\\
&\lesssim&
|\xi_1+\xi_2||\xi_1+\xi_3||\xi_1+\xi_4|/|\xi_1|^2,
\end{eqnarray*}
where $f(\xi)=m^2(\xi)\xi$. For (2), if $|\xi_1|\sim|\xi_4|$, it can be
show similarly as (1). If $|\xi_1|\gg|\xi_4|$,
we have $|\xi_1|-|\xi_3|\sim |\xi_1|$, thus
$$
|\xi_1+\xi_2||\xi_1+\xi_3||\xi_1+\xi_4|\sim|\xi_1+\xi_2||\xi_1|^2.
$$
Therefore, since $|\xi_1+\xi_2|\sim |\xi_3|$, we have,
\begin{eqnarray*}
|m^2(\xi_1)\xi_1+\cdots+m^2(\xi_4)\xi_4|
&\lesssim&
|\xi_1+\xi_2|+|\xi_3|+|\xi_4|\\
&\lesssim&
|\xi_1+\xi_2|
\lesssim
|\xi_1+\xi_2||\xi_1+\xi_3||\xi_1+\xi_4|/|\xi_1|^2.
\end{eqnarray*}
For (3), we can treat it as (2) when
$|\xi_1|\gg|\xi_4|$. Hence we have the claim by (ii).
This completes the proof of the lemma.
\hfill$\Box$
\end{proof}

Now we establish the following pointwise upper bound on the multiplier
$M_4$.
\begin{lem} let the $M_4$ defined in (\ref{M4}) and
$|\xi_{max}|=\max\{|\xi_1|,|\xi_2|,|\xi_3|\}$, we have
\begin{equation}\label{EM4}
    |M_4(\xi_1,\xi_2,\xi_3,\xi_4)|\lesssim
\min\left\{\frac{1}
{|\xi_{max}|},\frac{|\alpha_4(\xi_1,\xi_2,\xi_3,\xi_4)|}
{|\xi_1||\xi_2||\xi_3||\xi_4|}\right\}.
\end{equation}
\end{lem}
\begin{proof}
The first term $|M_4|\lesssim\dfrac{1}{|\xi_{max}|}$ easily follows from
(\ref{fact2}) and (\ref{EM3}). Now we turn to prove the second term.
Rewrite $M_4$ as
\begin{eqnarray}
M_4(\xi_1,\xi_2,\xi_3,\xi_4)
\!&\!=\!&\!
C\left[\frac{M_3(\xi_1,\xi_2,\xi_3+\xi_4)(\xi_3+\xi_4)}
{\alpha_3(\xi_1,\xi_2,\xi_3+\xi_4)}
+\frac{M_3(\xi_3,\xi_4,\xi_1+\xi_2)(\xi_1+\xi_2)}
{\alpha_3(\xi_3,\xi_4,\xi_1+\xi_2)}\right]_{sym}\nonumber\\
\!&\!=\!&\!
C\left[\left(\frac{M_3(\xi_1,\xi_2,\xi_3+\xi_4)}
{\alpha_3(\xi_1,\xi_2,\xi_3+\xi_4)}
-\frac{M_3(\xi_3,\xi_4,\xi_1+\xi_2)}
{\alpha_3(\xi_3,\xi_4,\xi_1+\xi_2)}\right)(\xi_3+\xi_4)\right]_{sym}.\label{4.12}
\end{eqnarray}
Moveover, it is easy to see that
$$
\frac{M_3(\xi_1,\xi_2,\xi_3+\xi_4)}
{\alpha_3(\xi_1,\xi_2,\xi_3+\xi_4)}
-\frac{M_3(\xi_3,\xi_4,\xi_1+\xi_2)}
{\alpha_3(\xi_3,\xi_4,\xi_1+\xi_2)}=I_1+I_2,
$$
for the $I_1,I_2$ defined as
\begin{eqnarray*}
  I_1 &=& M_3(\xi_1,\xi_2,\xi_3+\xi_4)\cdot
\left(\frac{1}
{\alpha_3(\xi_1,\xi_2,\xi_3+\xi_4)}
+\frac{1}
{\alpha_3(\xi_3,\xi_4,\xi_1+\xi_2)}\right);\\
  I_2 &=&
-\frac{m^2(\xi_1)\xi_1+m^2(\xi_2)\xi_2
          +m^2(\xi_3)\xi_3+m^2(\xi_4)\xi_4}
{\alpha_3(\xi_3,\xi_4,\xi_1+\xi_2)}.
\end{eqnarray*}
For $I_1$, by (\ref{fact2}), (\ref{EM3}) and Lemma 4.2(i), we have
\begin{eqnarray*}
|I_1|
&\lesssim&
\left|\xi_1+\xi_2\right|
\frac{|\alpha_3(\xi_1,\xi_2,\xi_3+\xi_4)+\alpha_3(\xi_3,\xi_4,\xi_1+\xi_2)|}
{|\alpha_3(\xi_1,\xi_2,\xi_3+\xi_4)\alpha_3(\xi_3,\xi_4,\xi_1+\xi_2)|}\\
&\lesssim&
\frac{|\alpha_4(\xi_1,\xi_2,\xi_3,\xi_4)|}
{|\xi_1||\xi_2||\xi_3||\xi_4||\xi_1+\xi_2|}.
\end{eqnarray*}
For $I_2$, , by Lemma 4.2(iii) and (\ref{fact2}), we also have
$$
|I_2|
\lesssim
\frac{|\alpha_4(\xi_1,\xi_2,\xi_3,\xi_4)|/|\xi_m|^2}
{|\alpha_3(\xi_1,\xi_2,\xi_3+\xi_4)|}\lesssim
\frac{|\alpha_4(\xi_1,\xi_2,\xi_3,\xi_4)|}
{|\xi_1||\xi_2||\xi_3||\xi_4||\xi_1+\xi_2|}.
$$
Then we have the result by combining (\ref{4.12}).
\hfill$\Box$
\end{proof}

Now we turn to give the pointwise upper bound on the multiplier
$M_5$. It directly follows from Lemma 4.3.
\begin{lem} For the $M_5$ defined in (\ref{M5}), we have
\begin{equation}\label{EM5}
    |M_5(\xi_1,\xi_2,\xi_3,\xi_4)|\lesssim
\frac{\chi_{\Omega_5}}{|\xi_1||\xi_2||\xi_3|}.
\end{equation}
where $\Omega_5=\{(\xi_1,\xi_2,\xi_3,\xi_4,\xi_5):
\xi_1+\xi_2+\xi_3+\xi_4+\xi_5=0,
  |\xi_1|, |\xi_2|,|\xi_3|,|\xi_4+\xi_5|\gtrsim N\}$.
\end{lem}

\subsection{Multilinear Estimates and Proof of Theorem 1.1}

First we give the comparison between $E^2_I(t)$ and $E^4_I(t)$.
\begin{lem} Let $I=I_{N,s}$ for $s\geq -\dfrac{3}{4}$, then
\beq
 |E^2_I(t)-E^4_I(t)|\lesssim
 N^{-\frac{3}{2}}\|Iu(t)\|_{L^2}^3+N^{-3}\|Iu(t)\|_{L^2}^4.
 \label{E23}
\eeq
\end{lem}
\begin{proof}
By the definitions (\ref{E3}) and (\ref{E4}), we need to show
\begin{eqnarray}
|\Lambda_3(\sigma_3)|
&\lesssim&
N^{-\frac{3}{2}}\|Iu(t)\|_{L^2}^3;\label{A3}\\
|\Lambda_4(\sigma_4)|
&\lesssim&
N^{-3}\|Iu(t)\|_{L^2}^4.\label{A4}
\end{eqnarray}
We may assume that $\F_xu(\xi,t)$ is nonnegative.
For (\ref{A3}), since $\xi_1+\xi_2+\xi_3=0$,
by symmetry we may assume again that
$|\xi_1|\sim|\xi_2|\geq |\xi_3|$. Note that
$\sigma_3$ vanishes when $|\xi_j|\leq N$ for $j=1,2,3$, so we may assume further that
$|\xi_1|,|\xi_2|\gtrsim N$.

Set
$$
\triangle\equiv\dfrac{|\sigma_3|}{m(\xi_1)m(\xi_2)m(\xi_3)}
=\dfrac{2\left|m^2(\xi_1)\xi_1+m^2(\xi_2)\xi_2+m^2(\xi_3)\xi_3\right|}
       {3|\alpha_3(\xi_1,\xi_2,\xi_3)|\,m(\xi_1)m(\xi_2)m(\xi_3)},
$$
then (\ref{E23}) follows if we show
$$
|\Lambda_3(\triangle)|\lesssim N^{-\frac{3}{2}}\|u\|_{L^2}^3.
$$

By (\ref{fact2}),  (\ref{EM3}) and $s\geq -\dfrac{3}{4}$, we have
$$
\triangle\lesssim\dfrac{1}{|\xi_1\xi_2|\,m(\xi_1)m(\xi_2)}\sim
N^{2s}|\xi_1|^{-1-s}|\xi_2|^{-1-s}\lesssim
N^{-\frac{3}{2}}|\xi_1|^{-\frac{1}{4}}|\xi_2|^{-\frac{1}{4}}.
$$
Therefore, by Plancherel's identity, H\"{o}lder and Sobolev's inequalities,
we have,
\beqs
\begin{split}
 |\Lambda_3(\triangle)|&\lesssim
N^{-\frac{3}{2}}|\Lambda_3(|\xi_1|^{-\frac{1}{4}}|\xi_2|^{-\frac{1}{4}})|\\
&\lesssim N^{-\frac{3}{2}}\left\|D_x^{-\frac{1}{4}}u\right\|_{L^4}^2\cdot
\|u\|_{L^2}\\
&\lesssim N^{-\frac{3}{2}}\|u\|_{L^2}^3.
\end{split}
\eeqs

Now we turn to (\ref{A4}).
Set
$$
\widetilde{\triangle}\equiv\dfrac{|\sigma_4|}{\prod\limits_{j=1}^4m(\xi_j)}
=\frac{|\tilde{M}_4(\xi_1,\cdots,\xi_4)|}
{|\alpha_4(\xi_1,\cdots,\xi_4)|\prod\limits_{j=1}^4m(\xi_j)},
$$
then (\ref{A4}) suffices if we show
$$
|\Lambda_4(\widetilde{\triangle})|
\lesssim
N^{-3}\>\|u(t)\|_{L^2}^4.
$$
Since $|\xi_j|\gtrsim N$ in $\Omega$ and $s\geq
-\dfrac{3}{4}$, by Lemma 4.3, we have
$$
\widetilde{\triangle}\lesssim
\dfrac{\chi_\Omega}{\prod\limits_{j=1}^4m(\xi_j)|\xi_j|}
\lesssim N^{-3}\prod\limits_{j=1}^4|\xi_j|^{-\frac{1}{4}}.
$$
Therefore, we have
$$
|\Lambda_4(\widetilde{\triangle})|\lesssim N^{-3}\left\|D^{-\frac{1}{4}}u(t)\right\|_{L^4_x}^4
\lesssim N^{-3}\|u(t)\|_{L^2}^4
$$
by Sobolev's inequality. \hfill$\Box$
\end{proof}

Next lemmas are the key estimates related to the almost conservation
of $E^4_I(t)$.

\begin{lem}
Let $I,s$ be as Lemma 4.5,  then
\beq
 \left|\displaystyle\int_0^\delta\Lambda_4(\bar{M}_4)\,dt\right|\lesssim
 N^{-3+}\>
 \|Iu\|_{X_{0,\frac{1}{2}+}^\delta}^4.
 \label{Lambda4}
\eeq
\end{lem}
\begin{proof}
By symmetry we may assume again that
$|\xi_1|\sim|\xi_2|\geq |\xi_3|$.
Since $\bar{M}_4=0$, when $|\xi_1|,\cdots, |\xi_4|\leq N$,
we may assume again that $|\xi_1|\sim |\xi_2|\gtrsim N$.
We always have $|\xi_4|\ll N$ in $\Gamma_4/\Omega$. To extend
the integration domain from $[0,\delta]$ to $\R$, we may need to borrow
$|\xi_1|^{0-}$ from the multiplier (see \cite{CKSTT5}, for the argument),
but this will not be mentioned since
it will only be recorded by $N^{0+}$ at the end.
Therefore, by Plancherel's identity, we only need to show
\begin{equation}\label{4.20}
\displaystyle\int_\ast
\frac{\bar{M}_4(\xi_1,\cdots, \xi_4)
\widehat{f_1}(\xi_1,\tau_1)\cdots \widehat{f_4}(\xi_4,\tau_4)}
{m(\xi_1)\cdots  m(\xi_4)}
\lesssim
N^{-3+}\>\|f_1\|_{X_{0,\frac{1}{2}+}}\cdots \|f_4\|_{X_{0,\frac{1}{2}+}},
\end{equation}
where $\displaystyle\int_\ast
=\int_{\stackrel{\xi_1+\cdots+\xi_4=0,}{ \tau_1+\cdots+\tau_{4}=0}}
\,d\xi_1 d\xi_2 d\xi_{3} d\tau_1 d\tau_2  d\tau_3$.

First, we note that $|\xi_1|-|\xi_3|\sim |\xi_1|.$
Otherwise, if $|\xi_3|=|\xi_1|+ o(|\xi_1|),$ then $|\xi_2|=|\xi_1|+ o(|\xi_1|),$
and $\xi_1\cdot \xi_2<0,\xi_2\cdot \xi_3>0$. Thus we have
$$
|\xi_1|=|\xi_2+\xi_3|+ o(|\xi_1|)=|\xi_2|+|\xi_3|+ o(|\xi_1|)=2|\xi_1|+ o(|\xi_1|),
$$
but it doesn't happen. Therefore, we have
$$
|\phi'(\xi_1)-\phi'(\xi_3)|\sim |\xi_1|^2,\quad
|\phi'(\xi_2)-\phi'(\xi_4)|\sim |\xi_2|^2.
$$
Thus, by Lemma 4.3 and using (\ref{IE}) two times,
the left-hand side of (\ref{4.20}) is controlled by
\begin{eqnarray*}
&&
 N^{2s}
 \displaystyle\int_\ast
 \frac{|\xi_1|^{-3-2s}}{m(\xi_3)}\cdot
 |\phi'(\xi_1)-\phi'(\xi_3)|^\frac{1}{2}|\phi'(\xi_2)-\phi'(\xi_4)|^\frac{1}{2}
 \widehat{f_1}(\xi_1,\tau_1)\cdots \widehat{f_4}(\xi_4,\tau_4)
 \nonumber\\
&\lesssim &
 N^{-3}
 \displaystyle\int_\ast
 |\phi'(\xi_1)-\phi'(\xi_3)|^\frac{1}{2}|\phi'(\xi_2)-\phi'(\xi_4)|^\frac{1}{2}
 f_1(\xi_1,\tau_1)\cdots f_4(\xi_4,\tau_4)
 \nonumber\\
&=&
 N^{-3}
 \displaystyle\int
 I^{\frac{1}{2}}(f_1,f_3)(x,t)\>
 I^{\frac{1}{2}}(f_2,f_4)(x,t)\,dx dt
 \nonumber\\
&\lesssim &
 N^{-3}
 \left\|I^{\frac{1}{2}}(f_1,f_3)\right\|_{L^2_{xt}}
 \left\|I^{\frac{1}{2}}(f_2,f_4)\right\|_{L^2_{xt}}^2
 \nonumber\\
&\lesssim &
 N^{-3}\>\|f_1\|_{X_{0,\frac{1}{2}+}}\cdots \|f_4\|_{X_{0,\frac{1}{2}+}}.
\end{eqnarray*}
This completes the proof of the lemma.   \hfill$\Box$
\end{proof}

\begin{lem}
Let $I,s$ be as Lemma 4.5, then
\begin{equation}
 \left|\displaystyle\int_0^\delta\Lambda_5(M_5)\,dt\right|\lesssim
 N^{-\frac{15}{4}+}\>
 \|Iu\|_{X_{0,\frac{1}{2}+}^\delta}^5.
 \label{Lambda4}
\end{equation}
\end{lem}
\begin{proof}
By the argument at the beginning of the proof of Lemma 4.6, we may
use Plancherel's identity and turn to show
\begin{equation}\label{4.22}
\displaystyle\int_{\ast}
\frac{M_5(\xi_1,\cdots, \xi_5)
\widehat{f_1}(\xi_1,\tau_1)\cdots \widehat{f_5}(\xi_5,\tau_5)}
{m(\xi_1)\cdots  m(\xi_5)}
\lesssim
N^{-\frac{15}{4}+}\>\|f_1\|_{X_{0,\frac{1}{2}+}}\cdots \|f_5\|_{X_{0,\frac{1}{2}+}},
\end{equation}
where $\displaystyle\int_\ast
=\int_{\stackrel{\xi_1+\cdots+\xi_5=0,}{ \tau_1+\cdots+\tau_5=0}}
\,d\xi_1\cdots d\xi_{4} d\tau_1\cdots  d\tau_4$.
By the definition of $\tilde{M}_4$,
we have: $|\xi_1|,|\xi_2|,|\xi_3|,$ $|\xi_4+\xi_5|\gtrsim N$. We may
assume that $|\xi_4|\geq |\xi_5|$ by symmetry. Now we split it into
two cases to analysis:
Case 1, $|\xi_4|, |\xi_5|\gtrsim N$;
Case 2, $|\xi_4|\gtrsim N\gg |\xi_5|$.

Case 1, $|\xi_4|, |\xi_5|\gtrsim N$.
By Lemma 4.4,
$$
\frac{M_5(\xi_1,\cdots, \xi_5)}
{ m(\xi_1)\cdots  m(\xi_5)}
\lesssim
N^{5s}|\xi_1|^{-1-s}|\xi_2|^{-1-s}|\xi_3|^{-1-s}|\xi_4|^{-s}|\xi_5|^{-s}
$$
Then, by Lemma 2.3 and note that $s\geq \dfrac{3}{4}$, the left-hand side of (\ref{4.22}) is bounded by
\begin{eqnarray*}
&&
 N^{5s}
 \displaystyle\int_\ast
 |\xi_1|^{-1-s}|\xi_2|^{-1-s}|\xi_3|^{-1-s}|\xi_4|^{-s}|\xi_5|^{-s}
 \widehat{f_1}(\xi_1,\tau_1)\cdots \widehat{f_5}(\xi_5,\tau_5)
 \\
&\lesssim &
 N^{-\frac{15}{4}}
 \left\|D_x^{-\frac{1}{4}}P^af_1\right\|_{L^4_x L^\infty_t}
 \cdots
 \left\|D_x^{-\frac{1}{4}}P^af_3\right\|_{L^4_x L^\infty_t}
 \left\|D_x^{\frac{3}{4}}P^af_4\right\|_{L^8_x L^2_t}
 \left\|D_x^{\frac{3}{4}}P^af_5\right\|_{L^8_x L^2_t}
 \\
&\lesssim &
 N^{-\frac{15}{4}}\>\|f_1\|_{X_{0,\frac{1}{2}+}}
 \cdots \|f_5\|_{X_{0,\frac{1}{2}+}}.
\end{eqnarray*}

Case 2, $|\xi_4|\gtrsim N\gg |\xi_5|$. In this case,
$$
\frac{M_5(\xi_1,\cdots, \xi_5)}
{ m(\xi_1)\cdots  m(\xi_5)}
\lesssim
N^{4s}|\xi_1|^{-1-s}|\xi_2|^{-1-s}|\xi_3|^{-1-s}|\xi_4|^{-s}.
$$
Then, by Lemma 2.4 and Lemma 2.5,
the left-hand side of (\ref{4.22}) is bounded by
\begin{eqnarray*}
&&
 N^{4s}
 \displaystyle\int_\ast
 |\xi_1|^{-1-s}|\xi_2|^{-1-s}|\xi_3|^{-1-s}|\xi_4|^{-s}
 \widehat{f_1}(\xi_1,\tau_1)\cdots \widehat{f_5}(\xi_5,\tau_5)
 \\
&\lesssim &
 N^{-4}
 \left\|f_1\right\|_{L^6_x L^6_t}
 \cdots
 \left\|f_3\right\|_{L^6_x L^6_t}
 \left\|I^{\frac{1}{2}}(f_4,f_5)\right\|_{L^2_x L^2_t}
 \\
&\lesssim &
 N^{-4}\>\|f_1\|_{X_{0,\frac{1}{2}+}}
 \cdots \|f_5\|_{X_{0,\frac{1}{2}+}}.
\end{eqnarray*}
This completes the proof of the lemma.   \hfill$\Box$
\end{proof}

Now we are ready to prove Theorem 1.1 by iteration.

Fix $N$ large and  depending on $\|u_0\|_{H^s}$. First of all, by
Proposition 3.1, (\ref{kbo})-(\ref{1.2}) is well-posed on $[0,\delta]$
in $H^s(\R)$ with
$$
\delta\sim\|I_{N,s}u_0\|^{-2-}_{L^2}\gtrsim N^{2s-}.
$$
Next, we turn to estimate $E^2_I(\delta)\equiv \|I_{N,s}u(\delta)\|_{L^2}^2$.
By (\ref{E3'}), Lemma 4.6 and Lemma 4.7, we have
\beq
 E^4_I(t)
 \lesssim
 E^4_I(0)+N^{-3+}\>\|Iu\|_{X_{0,\frac{1}{2}+}^\delta}^4
 +N^{-\frac{15}{4}+}\>\|Iu\|_{X_{0,\frac{1}{2}+}^\delta}^5, \quad t\in [0,\delta].
 \label{49}
\eeq
By Lemma 4.5, we have
\beq
 E^2_I(t)
 \lesssim
 E^4_I(t)+N^{-\frac{3}{2}}\|Iu(t)\|_{L^2}^3+N^{-3}\|Iu(t)\|_{L^2}^4,
 \label{410}
\eeq
and for $t=0$,
\beq
 E^4_I(0)
 \lesssim
 \|Iu_0\|_{L^2}^2+N^{-\frac{3}{2}}\|Iu_0\|_{L^2}^3+N^{-3}\|Iu_0\|_{L^2}^4.
 \label{411}
\eeq
Therefore, using (\ref{49})$\sim$(\ref{411}), (\ref{LSE}) and (\ref{II}), we have
\begin{eqnarray*}
 E^2_I(t)
& \lesssim&
   \|Iu_0\|_{L^2}^2+N^{-\frac{3}{2}}\|Iu_0\|_{L^2}^3+N^{-3}\|Iu_0\|_{L^2}^4
   +N^{-3+}\>\|Iu\|_{X_{0,\frac{1}{2}+}^\delta}^4
   \\
&&
   +N^{-\frac{15}{4}+}\>\|Iu\|_{X_{0,\frac{1}{2}+}^\delta}^5
   +N^{-\frac{3}{2}}\|Iu(t)\|_{L^2}^3+N^{-3}\|Iu(t)\|_{L^2}^4
   \\
& \lesssim&
   \|Iu_0\|_{L^2}^2
   +N^{-\frac{3}{2}}\|Iu_0\|_{L^2}^3+N^{-3}\|Iu_0\|_{L^2}^4
   +N^{-3+}\>\|Iu_0\|_{L^2}^4
   \\
&&
   +N^{-\frac{15}{4}+}\>\|Iu_0\|_{L^2}^5
   +N^{-\frac{3}{2}}\|Iu(t)\|_{L^2}^3+N^{-3}\|Iu(t)\|_{L^2}^4
   \\
& \leq&
   \dfrac{1}{4}C_0N^{-2s}
   +C_1(N^{-3+}N^{-4s}+N^{-\frac{15}{4}+}N^{-5s})
   \\
&&
   +C_2\left(N^{-\frac{3}{2}}\|Iu(t)\|_{L^2}^3+N^{-3}\|Iu(t)\|_{L^2}^4\right)
\end{eqnarray*}
for $t\in [0,\delta]$, where $C_0$ is the constant such that
$\|Iu_0\|_{L^2}^2\leq \dfrac{1}{8}C_0 N^{-2s}$. Therefore,
\begin{equation*}
     E^2_I(t)
 \leq
    \dfrac{1}{2}C_0N^{-2s}
    +C_2\left(N^{-\frac{3}{2}}\|Iu(t)\|_{L^2}^3+N^{-3}\|Iu(t)\|_{L^2}^4\right)
\end{equation*}
provided
\beq
      C_1(N^{-3+}N^{-4s}+N^{-\frac{15}{4}+}N^{-5s})
 \leq
      \dfrac{1}{4}C_0N^{-2s}.
      \label{condition1}
\eeq
So, for large $N$, it's easy to see that
$$
     E^2_I(t)\leq  C_0N^{-2s},\qquad t\in [0,\delta].
$$
In particular, $E^2_I(\delta)\leq C_0N^{-2s}$. Therefore,
by taking $u(\delta)$ as a new initial data and employing Proposition 3.1,
we can extend the
solution to $[0,2\delta]$ under the condition (\ref{condition1}).

Repeating this process $k$ times, then
$E^2_I(k\delta)\leq C_0N^{-2s}$ provided
\beq
 kC_1(N^{-3+}N^{-4s}+N^{-\frac{15}{4}+}N^{-5s})
 \leq
 \dfrac{1}{4}C_0N^{-2s}.
 \label{condi}
\eeq
Set $T=k\delta$, then (\ref{condi}) becomes
\beq
T\cdot
C_1
\delta^{-1}
(N^{-3+}N^{-4s}+N^{-\frac{15}{4}+}N^{-5s})
\leq
\dfrac{1}{4}C_0N^{-2s}.
\label{condition2}
\eeq
Therefore, for a given $T>0$,
the solution can be extended to $[0,T]$ if (\ref{condition2}) holds.
Choosing $N^{0+}\gtrsim T$, (\ref{condition2}) becomes
$$
\delta^{-1}(N^{-3+}N^{-4s}+N^{-\frac{15}{4}+}N^{-5s})\lesssim
N^{-2s}.
$$
Since $\delta^{-1}\lesssim N^{-2s+}$, the above inequality amounts
$$
-2s-3-4s< -2s, \quad -2s-\frac{15}{4}-5s< -2s,
$$
that is, $s>-\dfrac{3}{4}$.
This completes the proof of Theorem 1.1.

\end{document}